\documentclass[10pt,journal,letterpaper,twosides,onecolumn]{IEEEtran}

\usepackage{amsmath,amssymb,amscd,latexsym,dsfont,wasysym,bbm}
\usepackage{psfrag}
\usepackage{cite}
\usepackage{cases}
\usepackage{amsthm}
 \usepackage{setspace}
\usepackage{pgfplots} 
\newcommand{\tr}[1]{\textrm{#1}}
\newcommand{\mr}[1]{\mathrm{#1}}
\newcommand{\tnr}[1]{{\textnormal{#1}}}

\newcommand{\mf}[1]{\mathsf{#1}}
 
\newcommand{\ms}[1]{\mathds{#1}}

\newcommand{\ov}[1]{\overline{#1}}

\newcommand{\bi}{\boldsymbol{i}}

\newcommand{\bx}{\boldsymbol{x}}

\newcommand{\btheta}{\boldsymbol{\theta}}

\newcommand{\figref}[1]{Fig.~\ref{#1}}

\newcommand{\secref}[1]{Sec.~\ref{#1}}

\newcommand{\tabref}[1]{Table~\ref{#1}}

\newcommand{\ie}{i.e.,~} 		
\newcommand{\eg}{e.g.,~}	

\newcommand{\argmin}{\mathop{\mr{argmin}}}

\newcommand{\set}[1]{\{#1\}}
\newcommand{\SET}[1]{\left\{#1\right\}}

\newcommand{\cd}{\cdot}
\newcommand{\ld}{\ldots}

\newcommand{\e}{\mr{e}}

\newcommand{\PR}[1]{\Pr\SET{#1}}       	

\newcommand{\IND}[1]{\ms{I}\big[{#1}\big]}   	
\newcommand{\Ex}{\ms{E}}     			
\newcommand{\T}{^{\mr{T}}}            		
\newcommand{\dd}{\,\mr{d}}             		

\newcommand{\mfA}{\mf{A}}
\newcommand{\mfD}{\mf{D}}
\newcommand{\mfH}{\mf{H}}

\newcommand{\mfT}{\mf{T}}
\newcommand{\mfW}{\mf{W}}

\newcommand{\PhiW}{\Phi_{\mfW}} 
\newcommand{\PhiT}{\Phi_{\mfT}} 
\newcommand{\PhiH}{\Phi_{\mfH}} 
\newcommand{\PhiA}{\Phi_{\mfA}} 
\newcommand{\PhiD}{\Phi_{\mfD}} 

\usepackage[authoryear,comma,longnamesfirst,sectionbib,round]{natbib} 

\pgfplotsset{compat=1.12}

\usepackage[acronym,nonumberlist]{glossaries} 
\usepackage{url}

\newacronym[\glsshortpluralkey=PDFs,\glslongpluralkey=probability density functions]{pdf}{PDF}{probability density function}
\newacronym[\glsshortpluralkey=CDFs,\glslongpluralkey=cumulative density functions]{cdf}{CDF}{cumulative density function}
\newacronym[\glsshortpluralkey=CCDFs,\glslongpluralkey=complementary cumulative density functions]{ccdf}{CDF}{complementary cumulative density function}
\newacronym[\glsshortpluralkey=PMFs,\glslongpluralkey=probability mass functions]{pmf}{PMF}{probability mass function}
\newacronym[]{lhs}{l.h.s.}{left-hand side}
\newacronym[]{rhs}{r.h.s.}{right-hand side} 

\newacronym[]{bicm}{BICM}{bit-interleaved coded modulation}
\newacronym[]{bicmid}{BICM-ID}{BICM with iterative demapping}
\newacronym[]{cm}{CM}{coded modulation}
\newacronym[]{tcm}{TCM}{trellis-coded modulation}
\newacronym[]{mlc}{MLC}{multi-level coding}
\newacronym[]{pam}{PAM}{pulse amplitude modulation}
\newacronym[]{bpsk}{BPSK}{binary phase shift keying}
\newacronym[]{qam}{QAM}{quadrature amplitude modulation}
\newacronym[]{16qam}{16-QAM}{16-points quadrature amplitude modulation}
\newacronym[]{psk}{PSK}{phase shift keying}
\newacronym[\glsshortpluralkey=LLRs,\glslongpluralkey=logarithmic likelihood ratios]{llr}{LLR}{logarithmic likelihood ratio}
\newacronym[]{oc}{OC}{operating characteristic}
\newacronym[]{map}{MAP}{maximum a posteriori}
\newacronym[]{ml}{ML}{maximum likelihood}
\newacronym[]{ep}{EP}{expectation propagation}
\newacronym[\glsshortpluralkey=MIs,\glslongpluralkey=mutual informations]{mi}{MI}{mutual information}
\newacronym[\glsshortpluralkey=GMIs,\glslongpluralkey=generalized mutual informations]{gmi}{GMI}{generalized mutual information}
\newacronym[]{eesm}{EESM}{exponential effective-SNR-mapping}
\newacronym[]{bicm-gmi}{BICM-GMI}{BICM generalized mutual information}
\newacronym[]{awgn}{AWGN}{additive white Gaussian noise}
\newacronym[]{bsc}{BSC}{binary symetric channel}
\newacronym[]{amc}{AMC}{adaptive modulation and coding}
\newacronym[]{csi}{CSI}{channel state information}
\newacronym[]{cqi}{CQI}{channel quality indicator}

\newacronym[]{sp}{SP}{set-partitioning}
\newacronym[]{gsm}{GSM}{global system for mobile communications}
\newacronym[]{edge}{EDGE}{enhanced data rates for GSM evolution}
\newacronym[]{3gpp}{3GPP}{3rd generation partnership project}
\newacronym[]{umts}{UMTS}{Universal Mobile Telecommunication System}
\newacronym[]{lte}{LTE}{Long Term Evolution}
\newacronym[]{dvb}{DVB}{digital video broadcasting}
\newacronym[]{fdd}{FDD}{Frequency Division Duplexing}

\newacronym[\glsshortpluralkey=CCs,\glslongpluralkey=convolutional codes]{cc}{CC}{convolutional code}
\newacronym[\glsshortpluralkey=PCCCs,\glslongpluralkey=parallel concatenated convolutional codes]{pccc}{PCCC}{parallel concatenated convolutional code}
\newacronym[\glsshortpluralkey=TCs,\glslongpluralkey=turbo codes]{tc}{TC}{turbo code}
\newacronym{ldpc}{LDPC}{low-density parity-check}
\newacronym[]{ofdm}{OFDM}{orthogonal frequency-division multiplexing}
\newacronym[]{bep}{BEP}{bit-error probability}
\newacronym[]{wep}{WEP}{word-error probability}
\newacronym[]{sep}{SEP}{symbol-error probability}
\newacronym[]{pep}{PEP}{pairwise-error probability}
\newacronym[]{ttcm}{TTCM}{turbo-trellis coded modulation}
\newacronym[]{uep}{UEP}{unequal error protection}
\newacronym[\glsshortpluralkey=CENCs,\glslongpluralkey=convolutional encoders]{cenc}{CENC}{convolutional encoder}
\newacronym[]{mimo}{MIMO}{multiple-input multiple-output}
\newacronym[\glsshortpluralkey=SNRs,\glslongpluralkey=signal-to-noise ratios]{snr}{SNR}{signal-to-noise ratio}
\newacronym[\glsshortpluralkey=SINRs,\glslongpluralkey=the signal-to-interference-plus-noise ratios]{sinr}{SINR}{the signal-to-interference-plus-noise ratio}
\newacronym[]{msb}{MSB}{most-significative bit}
\newacronym[]{bcjr}{BCJR}{Bahl--Cocke--Jelinek--Raviv}
\newacronym[\glsshortpluralkey=SEDs,\glslongpluralkey=squared Euclidean distances]{sed}{SED}{squared Euclidean distance}
\newacronym[\glsshortpluralkey=EDs,\glslongpluralkey=Euclidean distances]{ed}{ED}{Euclidean distance}
\newacronym[\glsshortpluralkey=MEDs,\glslongpluralkey=minimum Euclidean distances]{med}{MED}{minimum Euclidean distance}
\newacronym[]{core}{CoRe}{constellation rearrangement}
\newacronym[]{msd}{MSD}{multistage decoding}
\newacronym[]{pdl}{PDL}{parallel decoding of the individual levels}
\newacronym[\glsshortpluralkey=GCs,\glslongpluralkey=Gray codes]{gc}{GC}{Gray code}
\newacronym[]{brgc}{BRGC}{binary-reflected Gray code}
\newacronym[]{nbc}{NBC}{natural binary code}
\newacronym[]{fbc}{FBC}{folded-binary code}
\newacronym[]{bsgc}{BSGC}{binary semi-Gray code}
\newacronym[]{msp}{MSP}{modified set-partitioning}
\newacronym[]{ssp}{SSP}{semi set-partitioning}
\newacronym[]{fhd}{FHD}{free Hamming distance}
\newacronym[]{mfhd}{MFHD}{maximum free Hamming distance}
\newacronym[]{ods}{ODS}{optimal distance spectrum}
\newacronym[]{iud}{i.u.d.}{independent and uniformly distributed}
\newacronym[]{ud}{u.d.}{uniformly distributed}
\newacronym[]{iid}{i.i.d.}{independent, identically distributed}
\newacronym[]{ami}{AMI}{accumulated mutual information}
\newacronym[]{bico}{BICO}{binary-input continuous-output}
\newacronym[]{gh}{GH}{Gauss--Hermite}
\newacronym[]{gum}{GUM}{Gaussian--uniform mixture}

\newacronym[\glsshortpluralkey=BSs,\glslongpluralkey=base-stations]{bs}{BS}{base-station}
\newacronym[\glsshortpluralkey=MSs,\glslongpluralkey=mobile-stations]{ms}{MS}{mobile-stations}

\newacronym[]{phy}{PHY}{physical layer} 
\newacronym[]{rlc}{RLC}{Radio-Link control} 
\newacronym[]{ran}{RAN}{Radio Access Network} 
\newacronym[]{llc}{LLC}{logical link control} 
\newacronym[]{tcp}{TCP}{transmission control protocol} 
\newacronym[]{mac}{MAC}{media access control} 
\newacronym[]{fft}{FFT}{fast Fourier transform} 
\newacronym[]{ft}{FT}{Fourrier transform}
\newacronym[]{cf}{CF}{characteristic function} 
\newacronym[]{mgf}{MGF}{moment generating function} 
\newacronym[]{ee}{EE}{energy efficiency} 
\newacronym[]{eb}{EB}{energy per bit}
\newacronym[]{kkt}{KKT}{Karush--Kuhn--Tucker} 
\newacronym[]{mcs}{MCS}{modulation/coding scheme} 
\newacronym[]{fec}{FEC}{forward error correction}
\newacronym[]{arq}{ARQ}{automatic repeat request}
\newacronym[]{harq}{HARQ}{hybrid ARQ}
\newacronym[]{tarq}{TARQ}{truncated HARQ}
\newacronym[]{ir}{IR}{incremental redundancy}
\newacronym[]{rpr}{RR}{repetition redundancy}
\newacronym[]{rrharq}{RR-HARQ}{repetition redundancy HARQ}
\newacronym[]{irharq}{IR-HARQ}{incremental redundancy HARQ}
\newacronym[]{ack}{ACK}{positive acknowledgment}
\newacronym[]{nack}{NACK}{negative acknowledgment}
\newacronym[]{hol}{HoL}{head of the line}
\newacronym[]{crc}{CRC}{cyclic redundancy check}
\newacronym[]{dp}{DP}{dynamic programming}
\newacronym[]{gp}{GP}{geometric programming}
\newacronym[]{per}{PER}{packet error rate}
\newacronym[]{ber}{BER}{bit error rate}
\newacronym[]{op}{OP}{outage probability}
\newacronym[]{spa}{SPA}{saddle-point approximation}
\newacronym[]{mrc}{MRC}{maximum ratio combining}
\newacronym[]{mdp}{MDP}{Markov decision process}
\newacronym[]{lp}{LP}{linear programming}
\newacronym[]{pomdp}{POMDP}{partially observable Markov decision process}
\newacronym[]{psimdp}{PSI-MDP}{partial state information Markov decision process}
\newacronym[]{scpp}{SCPP}{stochastic shortest path problem}

\newacronym[]{forw}{frwd}{forward}
\newacronym[]{feed}{fdbk}{feedback}

\newacronym[]{mm}{MM-HARQ}{multi-message HARQ}
\newacronym[]{xp}{XP-HARQ}{cross-packet HARQ}
\newacronym[]{ts}{TS}{time-sharing}
\newacronym[]{sc}{SC}{superposition coding}
\newacronym[]{sbrq}{SBRQ}{systematic backward retransmission}
\newacronym[]{brq}{BRQ}{backward retransmission}
\newacronym[]{lharq}{L-HARQ}{layer-coded HARQ}
\newacronym[]{anlharq}{AoN-HARQ}{all-or-none L-HARQ}
\newacronym[]{vlharq}{VL-HARQ}{variable-length HARQ}

\newacronym[]{pp}{PPP}{point process}
\newacronym[]{ppp}{PPP}{Poisson point process}
\newacronym[]{pgfl}{PGFL}{Poisson point process}

\newacronym[]{fide}{FIDE}{F\'ed\'eration Internationale des \'Echecs}
\newacronym[]{fifa}{FIFA}{F\'ed\'eration Internationale de Football Association}
\newacronym[]{sg}{SG}{stochastic gradient}
\newacronym[]{lms}{LMS}{least mean squares}
\newacronym[]{rls}{RLS}{recursive least squares}
\newacronym[]{vss}{VSS}{variable step-size}

\newtheorem{proposition}{Proposition}

\begin{document}

\title{Understanding and Pushing the Limits of  the Elo Rating Algorithm}
\author{Leszek Szczecinski and Aymen Djebbi
\thanks{%
L.~Szczecinski and A. Djebbi are with INRS, Montreal, Canada. [e-mail: \{leszek,~aymen.djebbi\}@emt.inrs.ca].  The work was supported by NSERC, Canada.}%
}%


\maketitle
\thispagestyle{empty}

\setstretch{1.6}

\begin{abstract}
This work is concerned with the rating of players/teams in face-to-face games with three possible outcomes: loss, win, and draw. This is one of the fundamental problems in sport analytics, where the very simple and popular, non-trivial algorithm was proposed by Arpad Elo in late fifties to rate chess players. In this work we explain the mathematical model underlying the Elo algorithm and, in particular, we explain what is the implicit but not yet spelled out, assumption about the model of draws. We further extend the model to provide flexibility and remove the unrealistic implicit assumptions of the Elo algorithm. This yields the new rating algorithm, we call $\kappa$-Elo, which is equally simple as the Elo algorithm but provides a possibility to adjust to the frequency of draws. The discussion of the importance of the appropriate choice of the parameters is carried out and illustrated using results from English Premier League football seasons.
\end{abstract}

\section{Introduction}\label{Sec:Intro}
Rating of players/teams is, arguably one of the most important issues in sport/ competition analytics.  In this work we are concerned with rating of the players/teams in the sports with one-on-one games yielding ternary results of win, loss and draw; such a situation appear in almost all team sports and many individual sports/competitions. 

Rating in sports consists in assigning a numerical value to a player/team using the results of the past games. While most of the sports' ratings use the points which are attributed to the game's winner, the rating algorithm which was developed in late fifties by Arpad Elo in the context of chess competition \citep{Elo08_Book}, and adopted later by \gls{fide}, challenged this view. 

Namely, the Elo algorithm changes the players' rating using not only the game outcome but also the ratings of the players before the game. The Elo algorithm is arguably one of the most popular, non-trivial rating algorithm and was used to analyze different sports, although mostly informally \citep[Chap.~5]{Langeville12_book}\citep{wikipedia_elo}; it is also used for rating in eSports \citep{Herbrich06}. Moreover, in 2018,  the Elo algorithm was adopted, under the name ``SUM'', by \gls{fifa} for the rating of the national football teams \citep{fifa_rating}. The Elo algorithm thus deserves a particular attention particularly because it is often presented without mathematical details behind its derivation, which may be quite confusing.

In this work we adopt the probabilistic modelling point of view, where the game outcomes are related to the rating by conditional probabilities. The advantage is that, to find the rating, we can use the conventional estimation strategies, such as  \gls{ml}; moreover, with the well defined model, once the ratings are found, they can be used for the purpose of prediction, which we understand as defining the distribution over the results of the game to come. This well-known mathematical formalism of rating in sport has been developed in psychometrics for rating the preferences in pairwise-comparison setup \citep{Thurston27}\citep{Bradley52} and the problem was deeply studied and extended in different directions \eg \citep{Cattelan12}\citep{Caron12}.

In this work we are particularly concerned with the mathematical modelling of draws (or ties). This issue has been addressed in psychometrics via two distinct approaches: in \citep{Rao67}, using thresholding of the unobserved (latent) variables and in \citep{Davidson70}--via an axiomatic approach. These two approaches have also been applied in sport rating, \eg \citep{Herbrich06}\citep{Joe90}; the former, however,  is used more often than the latter.

We note that the draws are not modelled in the Elo algorithm \citep{Elo08_Book}. In fact, and more generally, the outcomes are not explicitly modelled at all; rather, to derive the algorithm, the probabilistic model is combined with the strong intuition of the author; no formal optimality criteria is defined. Nevertheless, it was later observed that the Elo algorithm actually finds the approximate \gls{ml} ratings estimates in the binary-outcome (win-loss) games \citep{Kiraly17}. 

As for the draws, the Elo algorithm considers them by using the concept of a fractional score (of the game). However, since the underlying model is not specified, in our view there is a logical void: on the one hand, the Elo algorithm includes draws, on the other hand, there is no model allowing us to calculate the draw probability. The objective of this work is to fill this gap. 

The paper is organized as follows. We define the mathematical model of the problem in \secref{Sec:Model}. In \secref{Sec:rating.filtering} we show how the principle of \gls{ml} combined with the \gls{sg} yield the Elo algorithm in the binary-outcome games. We treat the issue of draws in \secref{Sec:Draws}; this is where the main contributions of the paper are found. Namely, we show and discuss the implicit model underlying the Elo algorithm; we also extend the model to increase its flexibility; finally we show how to define its parameters to take into account the known frequency of the draws. In \secref{Sec:Examples} we illustrate the analysis with numerical results and the final conclusions are drawn in \secref{Sec:Conclusions}.

\section{Rating: Problem definition}\label{Sec:Model}
We consider the problem of $M$ players (or teams), indexed by $m=1,\ld,M$, challenging each other in face-to-face games. At a time $n$ we observe the result/outcome $y_n$ of the game between the players defined by the pair $\bi_n=\set{i_{\tnr{H},n},i_{\tnr{A},n}}$. The index $i_{\tnr{H},n}$ refers to the ``home'' player, while $i_{\tnr{A},n}$ indicates  the ``away'' player. This distinction is often important in the team games where the so-called home-field advantage may play a role; in other competition such an effect may exist as well, like in chess, the player who starts the game may be considered a home player. We consider three possible game results: i)~the home player wins; denoted as $\set{i_{\tnr{H},n}\gtrdot  i_{\tnr{A},n}}$ in which case $\set{y_n=\mfH}$; ii)~the draw (or tie) $\set{y_n=\mfD}$, denoted also as $\set{i_{\tnr{H},n} \doteq i_{\tnr{A},n}}$; and finally, iii)~$\set{y_n=\mfA}$, which means that the ``away'' player wins which we  denote also as $\set{i_{\tnr{H},n} \lessdot  i_{\tnr{A},n}}$. 

For compactness of notation, useful in derivations, it is convenient to encode the categorical variable $y_n$ into numerical indicators defined over the set $\set{0,1}$
\begin{align}\label{omega.lambda.tau}
h_n&=\IND{y_n=\mfH},\quad  a_n=\IND{y_n=\mfA},\quad d_n=\IND{y_n=\mfD},
\end{align}
with $\IND{\cd}$ being  the indicator function: $\IND{A}=1$ if $A$ is true and $\IND{A}=0$, otherwise. The mutual exclusivity of the win/loss/draw events  guarantees $h_n+a_n+d_n=1$.

Having observed the outcomes of the games, $y_l, l=1,\ld,n$, we want to \emph{rate} the players, \ie assign a \emph{rating level}---a real number---$\theta_m$ to each of them. The rating level should represent the player's ability to win; for this reason it is also called \emph{strength} \citep{Glickman99} or \emph{skill} \citep{Herbrich06}\citep{Caron12}. The ability should be understood in the probabilistic sense: no player has a guarantee to win so the outcome $y_n$ is treated as a realization of a random variable $Y_n$. Thus, the levels $\theta_m, m=1,\ld, M$ should provide a reliable estimate of the distribution of $Y_n$ over the set $\set{\mfH,\mfA,\mfD}$. In other words, the formal rating becomes an expert system explaining the past-- and predicting the future results.
\subsection{Win-loss model}\label{Sec:rating.model}
It is instructive to consider first the case when the outcome of the game is binary, $y_n\in\set{\mfH, \mfA}$, \ie for the moment, we ignore the possibility of draws, $\mfD$, and we consider them separately in \secref{Sec:Draws}. In this case we are looking to establish the probabilistic model linking the result of the game and the rating levels of the involved players.  By far the most popular approach is based on the so-called linear model \citep[Ch.~1.3]{David63_Book}
\begin{align}\label{Pr.ij.PhiW.ov}
\PR{ i\gtrdot j |\theta_i,\theta_j} =  \PhiH(\theta_i-\theta_j),
\end{align}
where $\PhiH(v)$ is an increasing function which satisfies 
\begin{align}\label{}
\lim_{v\rightarrow-\infty}\PhiH(v)=0,\quad \lim_{v\rightarrow\infty}\PhiH(v)=1,
\end{align}
and thus we may set $\PhiH(v)=\Phi(v)$, where $\Phi(v)$ is  a conveniently chosen \gls{cdf}. By symmetry, $\PR{ i\gtrdot j }=\PR{ j\lessdot i }$, we obtain
\begin{align}\label{Pr.ij.PhiW}
\PhiH(v)=\Phi(v),\quad \PhiA(v) =  \Phi(-v)= 1-\Phi(v),
\end{align}
where the last relationship comes from the law of total probability, $\PR{i\gtrdot  j}+\PR{i\lessdot  j} =1$ (remember, we are dealing with binary-outcome games).

Indeed, \eqref{Pr.ij.PhiW.ov} corresponds to our intuition: the growing difference between rating levels $\theta_i-\theta_j$ should translate into increasing probability of user $i$ winning against the user $j$.

To emphasize that the entire model is defined by the \gls{cdf} $\Phi(v)$, which affects both $\PhiH(v)$ and $\PhiA(v)$ via \eqref{Pr.ij.PhiW}, we  keep the separate notation $\PhiH(v)$ and $\Phi(v)$ even if they are the same in the case we consider. 

A popular choice for $\Phi(v)$ is the logistic \gls{cdf} \citep[]{Bradley52}
\begin{align}\label{Phi.Logistic}
\Phi(v)=\frac{1}{1+10^{-v/\sigma}}=\frac{10^{0.5v/\sigma}}{10^{0.5v/\sigma}+10^{-0.5v/\sigma}},
\end{align}
where  $\sigma>0$ is a scale parameter.  

We note that the rating is arbitrary regarding 
\begin{itemize}
\item the origin---because any value $\theta_0$ can be added to all the levels $\theta_m$ without affecting the difference $v=\theta_i-\theta_j$ appearing as the argument of  $\Phi(\cd)$ in \eqref{Pr.ij.PhiW}, 
\item the scaling---because the levels $\theta_m$ obtained with the scale $\sigma$ can be transformed into levels $\theta'_m$ with a scale $\sigma'$ via multiplication: $\theta'_m=\theta_m\sigma' /\sigma$, and then the value of $\Phi(\theta_i-\theta_j)$ used with $\sigma$ is the same value as $\Phi(\theta'_i-\theta'_j)$ used with $\sigma'$;\footnote{The rating implemented by \gls{fifa} uses $\sigma=600$ \citep{fifa_rating}, while \gls{fide} uses $\sigma=400$} and 
\item the base of the exponent in \eqref{Phi.Logistic}; for example, $10^{-v/\sigma}=\e^{-v/\sigma'}$ with $\sigma'=\sigma \log_{10} \e$; therefore, changing from the base-$10$ to the base of the natural logarithm requires replacing $\sigma$ with $\sigma'$.
\end{itemize}
\section{Rating via maximum likelihood estimation}\label{Sec:rating.filtering}

Using the results from \secref{Sec:rating.model},  the random variables, $Y_n$, and the rating levels are related through conditional probability 
\begin{align}
\label{P.Yk.W}
\PR{  Y_n = \mfH | \bx_n, \btheta} &= \PhiH( v_n )=\Phi( v_n),\\
\label{P.Yk.L}
\PR{  Y_n = \mfA | \bx_n, \btheta} &= \PhiA(v_n) =\Phi( - v_n ),\\
\label{linear.combine}
v_n&=\bx_n \T \btheta = \theta_{i_{\tnr{H},n}}-\theta_{i_{\tnr{A},n}},
\end{align}
where $\btheta=[\theta_1,\ld,\theta_M]\T$ is the vector which gathers all the rating levels, $(\cd)\T$ denotes transpose,  $v_n$ is thus a result of linear combiner  $\bx_n$ applied to $\btheta$, and $\bx_n$ is the game-scheduling vector,  \ie
\begin{align}\label{}
\bx_n = [ 0, \ld, 0, \underbrace{1}_{i_{\tnr{H},n}\tr{-th pos.}},  0, \ld, 0, \underbrace{-1}_{i_{\tnr{A},n}\tr{-th pos.}}, 0, \ld, 0 ]\T.
\end{align}
We prefer the notation using the scheduling vector as it liberates us from somewhat cumbersome repetition of the indices $i_{\tnr{H},n}$ and $i_{\tnr{A},n}$ as in \eqref{linear.combine}.

Our goal now, is to find the levels $\btheta$ at time $n$ using the game outcomes $\set{y_l}_{l=1}^n$ and the scheduling vectors $\set{\bx_l}_{l=1}^n$. This is fundamentally a parameter estimation problem (model fitting) and we solve it using the \gls{ml} principle. The \gls{ml} estimate of $\btheta$ at time $n$ is obtained via optimization
\begin{align}\label{theta.ML}
\hat{\btheta}_n=\argmin_{\btheta} J_n(\btheta)
\end{align}
where 
\begin{align}
\label{Pr.y.theta}
J_n(\btheta)&=- \log \PR{\set{Y_l}_{l=1}^n=\set{y_l}_{l=1}^n|\btheta,\set{\bx_l}_{l=1}^n }. 
\end{align}
Further, assuming that conditioned on  the levels $\btheta$, the outcomes $Y_l$ are mutually independent, \ie $ \PR{\set{Y_l}_{l=1}^n=\set{y_l}_{l=1}^n|\btheta,\set{\bx_l}_{l=1}^n } =\prod_{l=1}^n\PR{Y_n=y_n|\btheta,\bx_n } $, we obtain
\begin{align}
J_n(\btheta)&= \sum_{l=1}^n  L_l(\btheta)\\
\label{F.k}
L_l(\btheta)&=-\log \PR{ Y_l=y_l | \btheta}\\
\label{F.k.2}
&=-h_l\log\PhiH(\bx_l\T\btheta)-a_l\log\PhiA(\bx_l\T\btheta),
\end{align}
where we applied the model \eqref{P.Yk.W}-\eqref{P.Yk.L}.

\subsection{Stochastic gradient and Elo algorithm}
The minimization in \eqref{theta.ML} can be done via steepest descent which would result in the following operations 
\begin{align}\label{Step.descent}
\hat{\btheta}_n \leftarrow \hat{\btheta}_n - \mu \nabla_{\btheta}J_n(\hat{\btheta}_n)
\end{align}
iterated (hence the symbol ``$\leftarrow$'') till convergence for a given $n$; the gradient is calculated as
\begin{align}\label{gradient.J.k}
 \nabla_{\btheta}J_n(\btheta) =  \sum_{l=1}^n \nabla_{\btheta}L_l(\btheta),
\end{align}
and the step, $\mu$, should be adequately chosen to guarantee the convergence. Moreover, since $J_n(\btheta)$ is convex,\footnote{The convexity comes from the fact that $-\log\PhiH(v)$ is convex in $v$ (easy to demonstrate by hand) and thus $\log\PhiH(\bx\T\btheta)$, being a concatenation of  a convex and  linear functions is also convex \citep[Appendix~A]{Tsukida11}.} the minimum is global.\footnote{While the minimum is global, it is not unique due to the ambiguity of the origin $\theta_0$ we mentioned at the end of \secref{Sec:rating.model}.}

From \eqref{F.k.2} we obtain
\begin{align}\label{gradient}
\nabla_{\btheta}L_l(\btheta) 
&= - h_l \bx_l \psi( \bx_l\T\btheta ) +a_l \bx_l \psi( -\bx_l\T\btheta )\\
\label{gradient.2}
&=-\bx_l e_l(v_l)
\end{align}
where, directly from \eqref{Phi.Logistic} we have
\begin{align}\label{psi.x}
\psi(v)=\frac{\dd }{\dd v}\log\Phi(v)=\frac{\Phi'(v)}{\Phi(v)}=\frac{1}{\sigma'}\Phi(-v),
\end{align}
where $\sigma'=\sigma \log_{10} \e$, and, using $\Phi(-v)=1-\Phi(v)$ we have
\begin{align}
\label{e.k}
e_l(v_l)
&=h_l\psi\big( v_l \big) +(h_l-1) \psi\big( -v_l \big)
\\
\label{e.k.2}
&=\frac{1}{\sigma'}[h_l -\Phi(v_l)].
\end{align}

The solution  obtained in \eqref{Step.descent} is based on the model \eqref{P.Yk.W}-\eqref{P.Yk.L} which requires $\btheta$ to remain constant throughout the time $l=1,\ld, n$. Since, in practice, the levels of the players may vary in time (the abilities evolve due to training, age, coaching strategies, fatigue, etc.), it is necessary to track  $\btheta$. 

To this end, arguably the simplest strategy relies on the \acrfull{sg} which differs from the steepest descent in the following elements: i)~at time $n$ only one iteration of the steepest descent is executed,  ii)~the gradient is calculated solely for the current observation term $L_n(\hat{\btheta}_n)$, and iii)~the available estimate $\hat{\btheta}_n$ is used as the starting point for the update
\begin{align}\label{stoch.gradient}
\hat{\btheta}_{n+1} &=  \hat{\btheta}_{n} -\mu\nabla_{\btheta} L_n(\btheta) =\hat{\btheta}_{n} + \mu \bx_n e_n(v_n)\\
\label{stoch.gradient.2}
&=\hat{\btheta}_{n} + \mu \bx_n [h_n-\Phi(v_n)]\\
\label{stoch.gradient.3}
&=\hat{\btheta}_{n} - \mu \bx_n [a_n-\Phi(-v_n)],
\end{align}
where $\mu$ is the adaptation step; with abuse of notation the fraction $\frac{1}{\sigma'}$ from \eqref{e.k.2} is absorbed by $\mu$ in \eqref{stoch.gradient.2}-\eqref{stoch.gradient.3}.


In the rating context, $\bx_n$ has only two non-zero terms, and therefore only the level of the players $i_{\tnr{H},n}$ and $i_{\tnr{A},n}$ will be modified. By inspection, the update \eqref{stoch.gradient.2}-\eqref{stoch.gradient.3} may be written as a single equation for any player $i\in\set{i_{\tnr{H},n},i_{\tnr{A},n}}$ 
\begin{align}\label{rating.SG}
\hat{\theta}_{n+1, i} &=\hat{\theta}_{n, i} + K\big[s_i-\Phi( \Delta_i )\big]
\end{align}
where $\Delta_i=\hat{\theta}_{n,i}-\hat{\theta}_{n,j}$ and $j$ is the index of the player opposing the player $i$,  \ie $j\neq i, j\in\set{i_{\tnr{H},n},i_{\tnr{A},n}}$; $s_i=\IND{i\gtrdot j}$ indicates if the player $i$ won the game. Since the variables $s_i$ and $\Delta_i$ are intermediary, on purpose we do not index them with $n$. 

We also replaced $\mu$ with $K$ so that \eqref{rating.SG} has the form of the Elo algorithm as usually presented in the literature \citep{Elo08_Book}\citep[Ch.~5]{Langeville12_book}. Thus, the Elo algorithm implements the \gls{sg} to obtain the \gls{ml} estimate of the levels $\btheta$ under the model \eqref{Pr.ij.PhiW}. This has been noted before, \eg in \citep{Kiraly17}. 

We also note that, in the description of the Elo algorithm  \citep{Elo08_Book}, $s_i$ is defined as a numerical ``score'' attributed to the game outcome $\mfH$ or $\mfA$. In a sense, it is a legacy of rating methods which attribute numerical value to the game result. On the other hand, in the modelling perspective we adopted, attribution of numerical values to the categorical variables $\mfH$ and $\mfA$ is not required.

\section{Draws}\label{Sec:Draws}

We want to address now the issue of draws (ties) in the game outcome. We ignored it for clarity of development, but draws are important results of the game and must affect the rating, especially in sports when they occur frequently, such as international football,  chess and many others sports and competitions \citep[Ch.~11]{Langeville12_book}. Some approaches in the literature go around this problem by ignoring the draws, other count them as partial wins/losses with fractional score $s_i=\frac{1}{2}$ \citep[Ch.~11]{Langeville12_book}\citep{Glickman15}. Such heuristics, while potentially useful, do not show explicitly how to predict the results of the games from the rating levels.

Thus, the preferred approach is to model the draws explicitly; we  must, therefore, augment our model to include the conditional probability of draws
\begin{align}\label{Pr.ij.PhiT}
\PR{ i\doteq j |\theta_i,\theta_j}=  \PhiD(\theta_i-\theta_j),
\end{align}
where by axiomatic requirement $\PhiD(v)$ should be decreasing with the absolute value of its argument, and be maximized for $v=0$. The justification is that large absolute difference in levels increases the probability of win or loss, while the rating levels proximity, $\theta_i\approx \theta_j$, should increase the probability of a draw.

By the law of total probability we require now
\begin{align}\label{sum.WTL}
\PhiH(v)+\PhiA(v) +\PhiD(v)=1,
\end{align}
which obviously implies that considering the draws, the functions $\PhiH(v)$ and $\PhiA(v)$ also must change with respect to those used when analyzing the binary (win/loss) game results.

\subsection{Explaining draws in the Elo algorithm}\label{Sec:draws.axiomatic}

The Elo algorithm also considers draws by setting $s_i=\frac{1}{2}$ in \eqref{rating.SG} \citep[Ch.~1.6]{Elo08_Book} \citep[Ch.~5]{Langeville12_book}. However, the function $\PhiD(v)$ is undefined which is quite perplexing: the draws are accounted for but the model, which would allow us to calculate their probability from the parameters $\btheta$, is lacking. Moreover, the description of algorithm  \eqref{rating.SG} still indicates that $\Phi(\Delta_i)$ is the ``expected score'' which cannot be calculated without explicit definition of the probability of draw. Despite this logical gap, the algorithm is being widely used and is considered reliable.

Our objective here, is thus to ``reverse-engineer'' the Elo algorithm and explain what probabilistic model is compatible with the operation of the algorithm. This will bridge the gap providing formal basis to interpret the results.

\begin{proposition}[The Elo algorithm with draws] The Elo algorithm \eqref{rating.SG} which assigns the score value $s_i=1$ to a win, $s_i=0$ to a loss and $s_i=\frac{1}{2}$ to a draw, implements \gls{sg} to estimate the rating levels $\btheta$ using the  \gls{ml} principle for model defined by the following conditional probabilities 
\begin{align}
\label{PhiW.WLT}
\PhiH(v)&=\Phi^2(v),\quad
\PhiA(v)=\Phi^2(-v)\\
\label{PhiT.WLT}
\PhiD(v)&=2\Phi(v)\Phi(-v).
\end{align}

\end{proposition}
\begin{proof}
We start by squaring the equation of the total probability law for the binary-outcome game,  $\Phi(v)+\Phi(-v)=1$, to obtain
\begin{align}\label{total.law.WLT}
\Phi^2(v)+\Phi^2(-v)+2\Phi(v)\Phi(-v)=1
\end{align}
and thus, using the assignment \eqref{PhiW.WLT}-\eqref{PhiT.WLT}, we satisfy \eqref{sum.WTL}. This may appear arbitrary but we have to recall that the whole model for the binary outcome is built on assumptions which reflect our idea about the loss/win probabilities and indeed, the draw probability function $\PhiD(v)$ has the behaviour we expected: it has a maximum for $v=0$ and decreases with growing $|v|$.

Now, each function in the model,  $\PhiH(v)$, $\PhiA(v)$, and $\PhiD(v)$, is a non-trivial transformation of $\Phi(v)$.

Using \eqref{PhiW.WLT}-\eqref{PhiT.WLT}, we rewrite \eqref{F.k} as
\begin{align}
\label{F.k.WTL}
L_l(\btheta)&=-\log \PR{ Y_l=y_l | \btheta}\\
\label{F.k.WTL.2}
&=-h_l\log\PhiH(v_l)-a_l\log\PhiA(v_l)-d_l\log\PhiD(v_l)\\
&=-2h_l\log\Phi(v_l)-2a_l\log\Phi(-v_l)-d_l\Big[\log\Phi(v_l) +\log\Phi(-v_l)\Big]
\end{align}
so the gradient is calculated as in \eqref{gradient}
\begin{align}\label{gradient.WLT}
\nabla_{\btheta}L_l(\btheta) 
&= - 2 \tilde{h}_l\bx_l \psi( v_l ) +2\tilde{a}_l\bx_l \psi( -v_l )\\
\label{gradient.WLT.2}
&=-2\bx_l \tilde{e}_l(\btheta)
\end{align}
where $\tilde{h}_l=h_l+d_l/2$, $\tilde{a}_l=a_l+d_l/2=1-\tilde{h}_l$, and 
\begin{align}
\label{e.k.WLT}
\tilde{e}_l(\btheta) &=\tilde{h}_l -\Phi(v_l).
\end{align}

We thus recover the same equations as in the binary-result game, splitting the draw indicator, $d_l$, equally between the indicators of the home and away wins; we can reuse them directly in \eqref{stoch.gradient.2}-\eqref{stoch.gradient.2}
\begin{align}\label{}
\label{stoch.gradient.2.HAD}
\hat{\btheta}_{n+1}&=\hat{\btheta}_{n} + \mu \bx_n [\tilde{h}_n-\Phi(v_n)]\\
\label{stoch.gradient.3.HAD}
&=\hat{\btheta}_{n} - \mu \bx_n [\tilde{a}_n-\Phi(-v_n)],
\end{align}
which yields the same update as the Elo algorithm \eqref{rating.SG} 
\begin{align}\label{rating.SG.WLT}
\hat{\theta}_{n+1, i} &= \hat{\theta}_{n, i} + K \big[s_i-\Phi( \Delta_i )\big],
\end{align}
with the new definition of the score $s_i=\tilde{h}_n$ (for the home player) and $s_i=\tilde{a}_n$ (for the away player), and where for compatibility of equations, the update step $K$ absorbed the multiplication by $2$ (the only difference between \eqref{gradient.WLT.2} and \eqref{gradient.2}). 
\end{proof}

The following observations are in order:
\begin{enumerate}
\item We unveiled the implicit model behind the Elo algorithm thus, our findings do not affect  the operation of the algorithm but rather clarify how to interpret its results. Namely, given the estimate of the levels $\hat{\btheta}_n$ the probability of the game outcomes should be estimated as
\begin{align}\label{Pr.win.hat.theta}
\PR{i\gtrdot j|\hat{\theta}_i,\hat{\theta}_j} &=\Phi^2 ( \hat{\theta}_i-\hat{\theta}_j )\\
\label{Pr.loss.hat.theta}
\PR{i\lessdot j|\hat{\theta}_i,\hat{\theta}_j}&=\Phi^2 ( \hat{\theta}_j-\hat{\theta}_i )\\
\label{Pr.draw.hat.theta}
\PR{i\doteq j|\hat{\theta}_i,\hat{\theta}_j}& =2\Phi ( \hat{\theta}_i-\hat{\theta}_j )\Phi( \hat{\theta}_j-\hat{\theta}_i ).
\end{align}
\item We emphasize that $s_i$ is the indicator of the result but  using \eqref{PhiW.WLT}-\eqref{PhiT.WLT} we can again calculate its expected value for $i=i_{\tnr{H},n}$
\begin{align}\label{}
\Ex_{Y_l|\hat{\theta}_{l,i},\hat{\theta}_{l,j}}[ s_i(Y_l) ] 
&=\PR{i\gtrdot j|\hat{\theta}_i,\hat{\theta}_j} +\frac{1}{2}\PR{i\doteq j|\hat{\theta}_i,\hat{\theta}_j}\\
&=\big[\Phi ( \Delta_i )\big]^2+\Phi ( \Delta_i )\Phi( -\Delta_i )\\
&=\Phi ( \Delta_i )\big[\Phi ( \Delta_i ) +\Phi ( -\Delta_i )\big]=\Phi ( \Delta_i );
\end{align}
the same can be straightforwardly done for $i=i_{\tnr{A},n}$.

Thus, indeed, the function $\Phi(\Delta_i)$ in the Elo update \eqref{rating.SG} has the meaning of the expected score. It has not been spelled out mathematically up to now---most likely---because the draws has been only implicitly considered. Nevertheless, with formidable intuition, the description of the Elo algorithm defines correctly the terms without making reference to the underlying probabilistic model. 

We note, again, that the notion of expected score is not necessary in the development of the \gls{sg} algorithm and the fact that the score takes the fractional value $s_i=\frac{1}{2}$ is a result of the particular form of the conditional probability \eqref{PhiT.WLT} and our decision to make $K$ absorb the multiplication by $2$, see \eqref{gradient.WLT.2}.
\end{enumerate}

While the clarification we made regarding the meaning of the expected score is useful, the first observation above is the most important for the explicit interpretation of the results of the algorithm. Recall that, in the win-loss game, the function $\Phi ( \Delta_i )$ has the meaning of the probability of winning the game, see \eqref{Pr.ij.PhiW}. However, in the win-draw-loss model, such interpretation is incorrect because the probability of winning the game is given by \eqref{Pr.win.hat.theta} which we just derived. As we will see in the numerical examples, using the latter, however, provides poor results. 

%

This surprising confusion persisted through time because the model we  have shown in \eqref{PhiW.WLT}-\eqref{PhiT.WLT} is merely implicit in the Elo algorithm and the explicit derivation of the algorithm \citep[Chap.~8]{Elo08_Book} did not consider the draws in the formal probabilistic framework. Other works, \eg \citep{Glickman99} \citep{Lasek13}, observed this conceptual difficulty before. In particular, \citep[Sec.~2]{Glickman99} used $\PhiT(v)= \sqrt{\Phi ( v )\Phi ( -v )}$ but kept the legacy of the win-loss model, \ie $\PhiH(v)= \Phi ( v )$ and $\PhiA(v)=\Phi(-v)$, which leads to approximate solutions because \eqref{sum.WTL} is violated.

The lesson learned is that, despite the apparent simplicity of the Elo algorithm, we should resist the temptation to tweak its parameters. While using the fractional score value $s_i=\frac{1}{2}$ for the draw is now explained, we cannot guarantee that modifying $s_i$  in arbitrary manner will correspond to a particular probabilistic model. Therefore, rather than tweaking the \gls{sg}/Elo algorithm \eqref{rating.SG}, the modification should start with the probabilistic model itself.

\subsection{Generalization of the Elo algorithm}\label{Sec:Extended.Elo}

Having unveiled the implicit modeling of draws underlying the Elo algorithm we immediately face a new problem. Namely, considering the draws, we have three events (and thus two independent probabilities to estimate) but the Elo algorithm has no additional degree of freedom to take this reality into account. For example, using \eqref{Pr.win.hat.theta}-\eqref{Pr.draw.hat.theta} the results of the game between the players with equal rating levels $\hat{\theta}_i=\hat{\theta}_j$, will always be predicated as $\PR{i\gtrdot j|\hat{\theta}_i,\hat{\theta}_j} =0.25$ and $\PR{i\doteq j|\hat{\theta}_i,\hat{\theta}_j} =0.5$. The Elo algorithm does that implicitly, but there is no real reason to stick to such a rigid solution which may produce an  inadequate fit to the observed data, and a more general approach is necessary.

One of the  workarounds proposed by  \citep{Rao67} and used later, \eg \citep{Fahrmeir94}\citep{Herbrich06}\citep{Kiraly17}, modifies the model using a threshold value $v_0\geq 0$
\begin{align}\label{Phi.WLT.thresholds}
\PhiH(v)&= \Phi(v- v_0),\quad \PhiA(v)= \Phi( -v - v_0), \quad \PhiD(v) = \Phi(v+v_0) -\Phi(v- v_0).
\end{align}

While \eqref{Phi.WLT.thresholds} is definitely useful and solves formally the problem which is more general than the case of binary outcome game, we do not treat it as a generalization of the Elo algorithm itself, because there is no parameter $v_0$ which transforms \eqref{Phi.WLT.thresholds} into \eqref{PhiW.WLT}-\eqref{PhiT.WLT} (which, as we demonstrated, is the model behind the Elo algorithm).

Here we propose to use the model of \citep{Davidson70} which can be defined as
\begin{align}
\label{PhiW.Davidson}
\PhiH(v)&= \Phi_{\kappa}(v) =   \frac{10^{0.5v/\sigma}}{10^{0.5v/\sigma}+10^{-0.5v/\sigma} +\kappa}\\
\label{PhiL.Davidson}
\PhiA(v)&=\Phi_{\kappa}(-v) =  \frac{10^{-0.5v/\sigma}}{10^{0.5v/\sigma}+10^{-0.5v/\sigma} +\kappa}\\
\label{PhiT.Davidson}
\PhiD(v)&=\kappa\sqrt{\PhiH(v)\PhiA(v)}=\frac{\kappa}{10^{0.5v/\sigma}+10^{-0.5v/\sigma} +\kappa},
\end{align}
where $\kappa\geq 0$ is a freely set draw parameter.

We hasten to say that the model  \eqref{PhiW.Davidson}-\eqref{PhiT.Davidson} is not necessarily better in the sense of fitting to the data than \eqref{Phi.WLT.thresholds}. Our motivation to adopt  \eqref{PhiW.Davidson}-\eqref{PhiT.Davidson} is the fact that these equations generalize previous models. Namely, for $\kappa=0$ we obtain the win-loss model behind the Elo algorithm shown in \eqref{rating.SG}, while using $\kappa=2$ yields
\begin{align}\label{}
\label{PhiW.Davidson.n=2}
\PhiH(v)&= \frac{10^{0.5v/\sigma}}{\Big(10^{0.25v/\sigma}+10^{-0.25v/\sigma}\Big)^2}=\Phi^2(v/2)
\end{align}
which, up to the scale factor $\sigma$, corresponds to the implicit win-draw-loss model behind the Elo algorithm we have shown in \eqref{PhiW.WLT}-\eqref{PhiT.WLT}. 

In other words, the implicit model for the Elo algorithm is based on the explicit modeling  of draws  proposed by \citep{Davidson70} if we set a particular value of the draw parameter ($\kappa=2$).

\subsubsection{Adaptation}\label{Sec:Adaptation.Extended.Elo}

We quickly note that the function $-\log \Phi_{\kappa}(v)$ is convex so the gradient-based adaptation will converge under adequate choice of the step $\mu$.

To derive the adaptation algorithm we recalculate \eqref{F.k.WTL} 
\begin{align}\label{}
L_l(\btheta)&=-\log \PR{ Y_l=y_l | \btheta}\\
&=-h_l\log\PhiH(v_l)-a_l\log\PhiA(v_l)-d_l\log\PhiD(v_l)\\
&=-\tilde{h}_l\log\PhiH(v_l)-\tilde{a}_l\log\PhiH(-v_l)
\end{align}
and the gradient is given by
\begin{align}\label{}
\nabla_{\theta}L_l(\btheta)
&=-e_l(v_l)\bx_l
\end{align}
where
\begin{align}\label{}
e_l(v_l)&=
\tilde{h}_l\psi_\kappa(v_l)+(\tilde{h}_l-1)\psi_\kappa(-v_l)\\
\psi_{\kappa}(v)&=\frac{\Phi'_{\kappa}(v)}{\Phi_{\kappa}(v)}
=\frac{1}{\sigma'}\frac{10^{-0.5v/\sigma}+\frac{1}{2}\kappa}{10^{0.5v/\sigma}+10^{-0.5v/\sigma} +\kappa}
=\frac{1}{\sigma'} F_{\kappa}(-v),
\end{align}
where, as before $\sigma'=\sigma\log_{ 10}\e$, and we define
\begin{align}\label{F.kappa.definition}
F_\kappa( v ) =\frac{10^{v/2}+\frac{1}{2} \kappa}{10^{v/2}+10^{-v/2} +\kappa}=1-F_\kappa( -v ),
\end{align}
and thus
\begin{align}\label{}
e_l(v_l)
&=
\frac{1}{\sigma'}\Big(\tilde{h}_l F_\kappa(-v_l) + (\tilde{h}_l-1)F _\kappa(v_l)\Big)\\
\label{e.l.GElo}
&=
\frac{1}{\sigma'}\Big(\tilde{h}_l-F_\kappa(v_l)\Big).
\end{align}

Using \eqref{e.l.GElo} in \eqref{stoch.gradient} yields the same equations as in \eqref{rating.SG.WLT} except that $\Phi(v)$ must be replaced with $F_\kappa(v)$ and the division by $\sigma'$ should be absorbed by the adaptation step. This yields a new $\kappa$-Elo rating algorithm
\begin{align}\label{rating.SG.WLT.Extended}
\hat{\theta}_{n+1, i} &= \hat{\theta}_{n, i} + K\big[s_i-F_\kappa( \Delta_i )\big],
\end{align}
where as before i)~$\Delta_i= \hat{\theta}_i - \hat{\theta}_j$ ($j$ being the index of the player opposing the player $i$), ii)~as in the Elo algorithm, $K$ is maximum increase/decrease step, and iii)~$s_i\in\set{0,\frac{1}{2},1}$ indicates the outcome of the game, \ie the score.


The new $\kappa$-Elo algorithm is equally simple as the Elo algorithm, yet provides us with the flexibility to model the relationship between the draws and the wins via the draw parameter $\kappa\geq0$. We recall that, in fact, the Elo algorithm is a particular version of $\kappa$-Elo for $\kappa=2$.
We provide numerical examples in \secref{Sec:Examples} to illustrate its properties. 


\subsubsection{$\kappa$ in $\kappa$-Elo algorithm: insights and pitfalls}\label{Sec:setting.kappa}
Can we say something about the draw parameter, $\kappa$, without implementing and running the $\kappa$-Elo algorithm  defined by \eqref{rating.SG.WLT.Extended}? The answer is yes, if we suppose that the fit we obtain is (almost) perfect, \ie the average empirical probabilities averaged over a large time window
\begin{align}\label{}
\ov{p}_\mfH&=\frac{1}{N}\sum_{l=1}^N\IND{y_l=\mfH}, \quad \ov{p}_\mfA=\frac{1}{N}\sum_{l=1}^N\IND{y_l=\mfA},\quad \ov{p}_\mfD=\frac{1}{N}\sum_{l=1}^N\IND{y_l=\mfD},
\end{align}
can be deduced from the functions \eqref{PhiW.Davidson}-\eqref{PhiT.Davidson}  using the estimated rating levels $\hat{\btheta}$.\footnote{Such statistics may be obtained from previous seasons. While they do not change drastically through  seasons and may be treated as a prior, in the case of on-line rating, they may also be estimated from the recent past. However, we do not follow this idea further in this work.} If this is the case they should stay in the relationship prescribed by the model \eqref{PhiT.Davidson}, \ie $\ov{p}_\mfD \approx \kappa\sqrt{\ov{p}_\mfH \ov{p}_\mfA}$. 

Denoting the difference between the frequency of home and away wins as  $\ov{\delta}=\ov{p}_\mfH- \ov{p}_\mfA$, and from the law of total (empirical) probability we obtain $\ov{p}_\mfH=\frac{1}{2}(1-\ov{p}_\mfD+\ov{\delta})$ and $\ov{p}_\mfA=\frac{1}{2}(1-\ov{p}_\mfD-\ov{\delta})$ from which
\begin{align}\label{}
\ov{p}_\mfD \approx \frac{\kappa}{2}\sqrt{(1-\ov{p}_\mfD)^2-\ov{\delta}^2}
\end{align}
and thus, for the relatively small values of home/away imbalance, \eg $\ov{\delta}< 0.2$ we can ignore the term $\ov{\delta}^2$ which allows us simply say what is implicit assumption about $\ov{p}_\mfD$ for arbitrary $\kappa$
\begin{align}\label{kappa.2.pt}
\ov{p}_\mfD &\approx \frac{\kappa}{2+\kappa}.
\end{align} 
Thus using $\kappa=2$ (as done implicitly in the current rating of \gls{fide} and \gls{fifa}), suggests that the $\ov{p}_\mfD\approx 0.5$. Since this is not the case in none of the competitions where these rating are used, we can expect that, when implementing the new rating algorithm with a more appropriate value of $\kappa$, \gls{fide} and \gls{fifa} will improve the fit to the results in the sense of better estimation of the probabilities of win, loss, and draw.

We can also estimate the suitable value of $\kappa$ as
\begin{align}\label{pt.2.kappa}
\ov{\kappa} &\approx \frac{2\ov{p}_\mfD}{1-\ov{p}_\mfD}.
\end{align}

For example,  using $\ov{p}_\mfD \approx 0.25$ (which was the average frequency of draws in English Premier Ligue football games over ten seasons, see \secref{Sec:Examples}) we would find  $\ov{\kappa}\approx 0.7$. 

Is this value acceptable? 

Before answering this question, we have to point to a particular problem that can arise in the modelling of the draws. Namely, the current formulations known in the literature (the threshold-based \eqref{Phi.WLT.thresholds} or the one we used \eqref{PhiW.Davidson}-\eqref{PhiT.Davidson}), do not explicitly constrain the relationship between the predicted \emph{values} of probabilities. Of course, we always keep the relationship $\PhiA(v)=\PhiH(-v)$. Thus, considering the case  $v_l=\hat{\theta}_i-\hat{\theta}_j=\epsilon$ (where $\epsilon>0$ is a small rating difference), we have $\PhiH(\epsilon)\ge\PhiA(\epsilon)$ and our intuition follows: it is more probable that a stronger home player wins than he looses. 

On the other hand, it is not clear what should be said about the probability of the draw in such a case. Should we expect the probability of draw to be larger than the probability of home/away win? For example, is it acceptable to obtain the values $\PhiH(\epsilon)=0.42$, $\PhiA(\epsilon)=0.38$ and $\PhiD(\epsilon)=0.20$? Nothing prevents such results in the model we use (and, to our knowledge, in other models used before) and the interpretation is counterintuitive: the stronger home player is more likely to loose than to draw.

Therefore, we might want to remove such results from the solution space: for equal-rating players we force the probability of the draw to be larger than the probability of home/away wins, we thus have to use $\kappa$ which satisfies 
\begin{align}\label{D.gt.W}
\PhiD(0)&>\PhiW(0)\\
\label{k.gt.1}
\kappa&\geq1,
\end{align}
where the last inequality follows from \eqref{PhiW.Davidson}-\eqref{PhiT.Davidson}. This is an important restriction and forces us to model the draws occurring with (a large) frequency $\ov{p}_\mfD\ge 0.33$, see \eqref{kappa.2.pt}. While it seems unsound to use the mismatched model, we don't know its impact on the prediction capability and yet, we have to remember, that the current version of the Elo algorithms uses $\kappa=2$. We have no clear answer to this question and will seek more insight in the numerical examples.


\section{Numerical Examples}\label{Sec:Examples}

We illustrate the operation of the algorithms using the results from the England Premier League football games available at \citep{football-data}. In this context, there are $M=20$ teams playing against each other in one home- and one away-games. We consider one season at the time, thus $n=1,\ld, N$, index the games in the chronological order, and $N=M(M-1)=380$. 

Football (and other) games are known to produce the so-called home-field advantage, where the home wins $\set{y_n=\mfH}$ are more frequent than the away wins $\set{y_n=\mfA}$. In the rating context, this is modelled by artificially increasing the level of the home player, which corresponds, de facto, to left-shifting of the conditional probability functions
\begin{align}\label{}
\PhiH^\tnr{hfa}(v)=\PhiH(v+\eta\sigma), \quad \PhiA^\tnr{hfa}(v)=\PhiA(v+\eta\sigma), \quad \PhiD^\tnr{hfa}(v)=\PhiD(v+\eta\sigma),
\end{align}
where home-field advantage parameter $\eta\ge 0$ should be adequately set; its value is independent of the scale thanks to multiplication by $\sigma$.\footnote{We note that this version of the equation is slightly different from \citep[Eq.~2.4]{Davidson77}; with our formulation, the relationship \eqref{pt.2.kappa} is not affected by the home-field advantage parameter $\eta$.}

As in FIFA rating algorithm, \citep{fifa_rating}, we set $\sigma=600$; the levels are initialized at $\theta_{0,m}=0$; as we said before these values are arbitrary.  In what follows we always use the normalization $K=\tilde{K}\sigma$ which removes the dependence on the scale: for a given $\tilde{K}$ the prediction results will be exactly the same even if we change the value of $\sigma$.

An example of the estimated ratings $\theta_{n,m}$ for a group of teams is shown in \figref{Fig:rating} to illustrate the fact that quite a large portion of time in the beginning of the season is dedicated to the convergences of the algorithm; this is the  ``learning'' period. Of course, using larger step $\tilde{K}$ we can accelerate the learning at the cost of increased variability of the rating. These well-known issues are related to the operation of \gls{sg} but solving them is out of the scope of this work. We mention them mostly because, to evaluate the performance of the algorithms, we decide to use the second half of the season, where we assume the algorithms converged and the rating levels follow the performance of the teams. This is somewhat arbitrary of course, but our goal here is to show the influence of  the draw-parameter and not to solve the entire problem of convergence/tracking in \gls{sg}/Elo algorithms.

\begin{figure}
\centering
\psfrag{xlabel}{\footnotesize  $n$}
\psfrag{2015}{}
\psfrag{ylabel}{\footnotesize $\hat{\theta}_{m,n}$}
\includegraphics[width=0.8\linewidth]{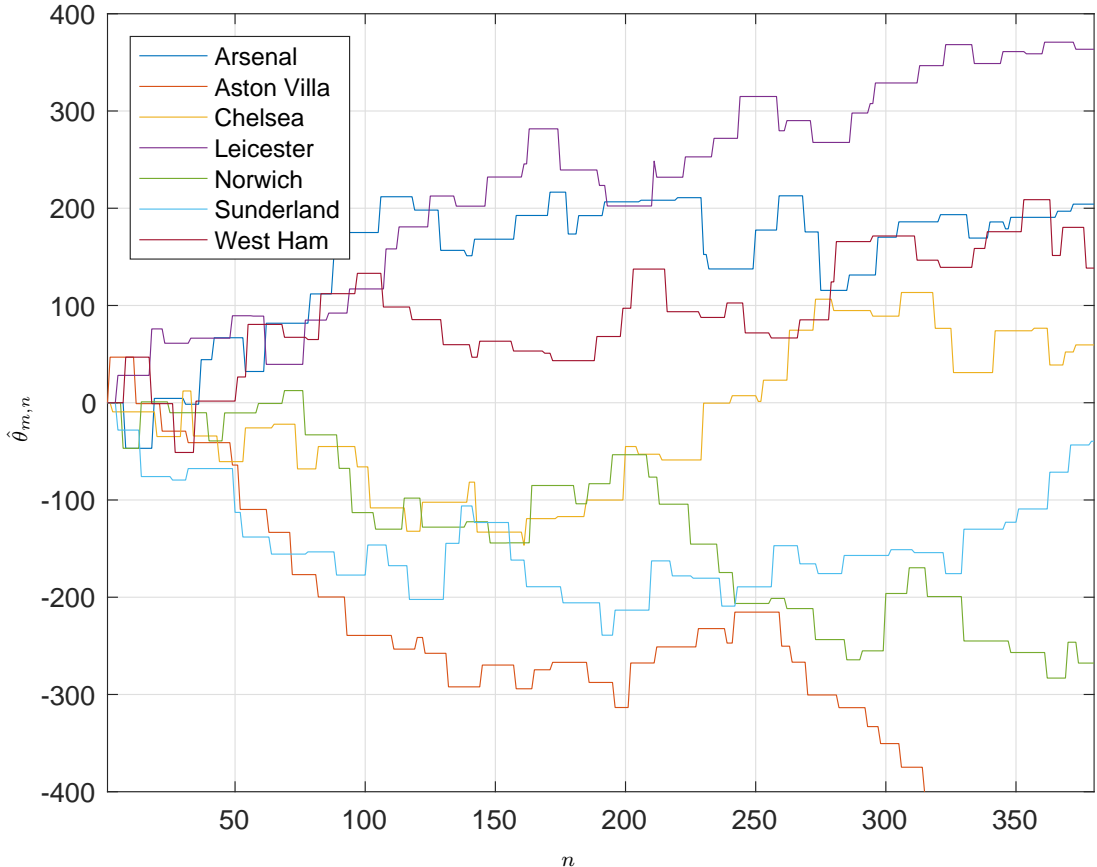}
\caption{Evolution of the rating levels $\hat{\theta}_{m,n}$ for selected English Premier League teams in the season 2015; $N=380$, $\sigma=600$, $\tilde{K}=0.125$, $\eta=0.3$, $\kappa=0.7$. We assume that, the first half of the season absorbs the learning phase, and the tracking of the teams' levels in the  second half is free of the initialization effect.}
\label{Fig:rating}
\end{figure}

For concision, the estimated probability of the game result  $\set{\mfH,\mfA, \mfD}$ calculated before the game at the time $l$ using the rating levels $\hat{\btheta}_{l-1}$  obtained at the time $l-1$, is denoted as
\begin{align}\label{}
\hat{p}_{l,\mfH}=\Phi_{\mfH}(\bx\T_l\hat{\btheta}_{l-1}), \quad \hat{p}_{l,\mfA}=\Phi_{\mfA}(\bx\T_l\hat{\btheta}_{l-1}), \quad \hat{p}_{l,\mfD}=\Phi_{\mfD}(\bx\T_l\hat{\btheta}_{l-1}).
\end{align}

We show the (negative) logarithmic score \citep{Gelman2014} averaged over the second-half of the season 
\begin{align}\label{log.score}
\ov{\tnr{LS}}= \frac{2}{N}\sum_{l=N/2+1}^{N} \tnr{LS}_{l}.
\end{align}
where
\begin{align}\label{}
\tnr{LS}_l &=  -( h_l \log \hat{p}_{l,\mfH} +a_l \log \hat{p}_{l,\mfA} +d_l \log \hat{p}_{l,\mfD}).
\end{align}

We still have to define the prediction of the draw in the conventional Elo algorithm: we cannot set $\hat{p}_{l,\mfD})=\PhiD(v)\equiv 0$, of course, because it would result in infinite logarithmic score. We thus follow the heuristics of \citep{Lasek13} which may be summarized as follows: the conventional Elo algorithm is used to find the rating levels (\ie $\kappa=2$ is used in $\kappa$-Elo), but the prediction is based on $\PhiH(v)$, $\PhiA(v)$, and $\PhiD(v)$ with a different value of the draw-parameter $\kappa=\check{\kappa}$. This may be seen as a model mismatch between estimation and prediction. We follow \citep{Lasek13} and apply $\check{\kappa}=1$; this correspond to $\ov{p}_\mfD\approx 0.33$ and also is the minimum value of $\kappa$ which guarantees \eqref{D.gt.W}.

We show in \figref{Fig:LogScore} the logarithmic score  $\ov{\tnr{LS}}$ for different values of the draw parameter $\kappa$, and normalized step $\tilde{K}$. We compare our predictions with those based on the probabilities inferred from the odds of the betting site Bets365 available, together with the game results, at \citep{football-data}.\footnote{This is done as in \citep{Kiraly17}: the published decimal odds for the three events, $o_\mfH$, $o_\mfA$, and $o_\mfD$, are used to infer the probabilities, $\tilde{p}_\mfH\propto 1/o_\mfH$, $\tilde{p}_\mfA\propto 1/o_\mfA$, and $\tilde{p}_\mfD\propto 1/o_\mfD$; these are next normalized to make them sum to one (required as the betting odds are not ``fair'' and include the bookie's overhead, the so-called vigorish).} These are constant reference lines in \figref{Fig:LogScore} as they, of course, do  not vary with the parameters we adjust.

We observe that introducing the draw parameter $\kappa$ we improved the logarithmic score. On the other hand, using $\kappa$-Elo algorithm with $\kappa=2$ yields particularly poor results if we explicit the model (and thus use $\kappa=2$ for the prediction); a much better solution is to use a mismatched model and apply $\check{\kappa}=1$; the results obtained are, in general very close to those obtained using $\kappa$-Elo algorithm especially when used with $\kappa=1$. In the season 2013-2014, where frequency of draws was low, using the corresponding value $\kappa=\ov{\kappa}$ provided notable improvement comparing to large $\kappa\in\set{1,2}$. 

\begin{figure}
\centering
\psfrag{score}{\footnotesize$\ov{\tnr{LS}}$}
\psfrag{0.40}{\tiny $\kappa=0.4$}
\psfrag{0.70}{\tiny $\kappa=0.7$}
\psfrag{1.00}{\tiny $\kappa=1$}
\psfrag{2.00}{\tiny $\kappa=2$}
\psfrag{Eloxxxxxx}{\tiny Elo+$\check{\kappa}$}
\psfrag{Bet365}{\tiny Bet365}
\psfrag{2017}[cb]{\footnotesize $2017-2018$}
\psfrag{2013}[cb]{\footnotesize $2013-2014$}
\begin{tabular}{cc}
\psfrag{xlabel}[tt]{\footnotesize $\eta$}
\includegraphics[width=0.45\linewidth]{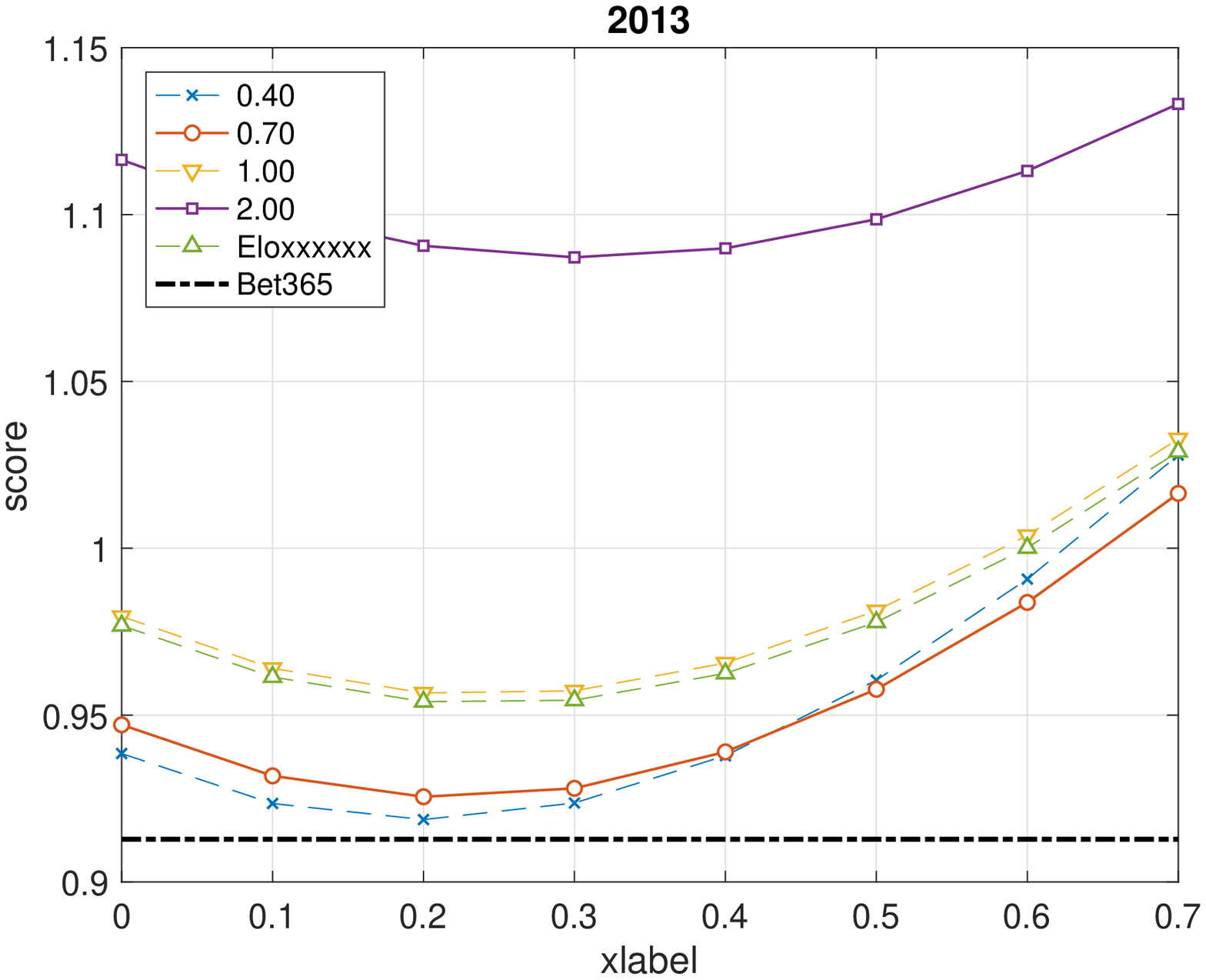} &
\psfrag{xlabel}[tt]{\footnotesize $\eta$}
\includegraphics[width=0.45\linewidth]{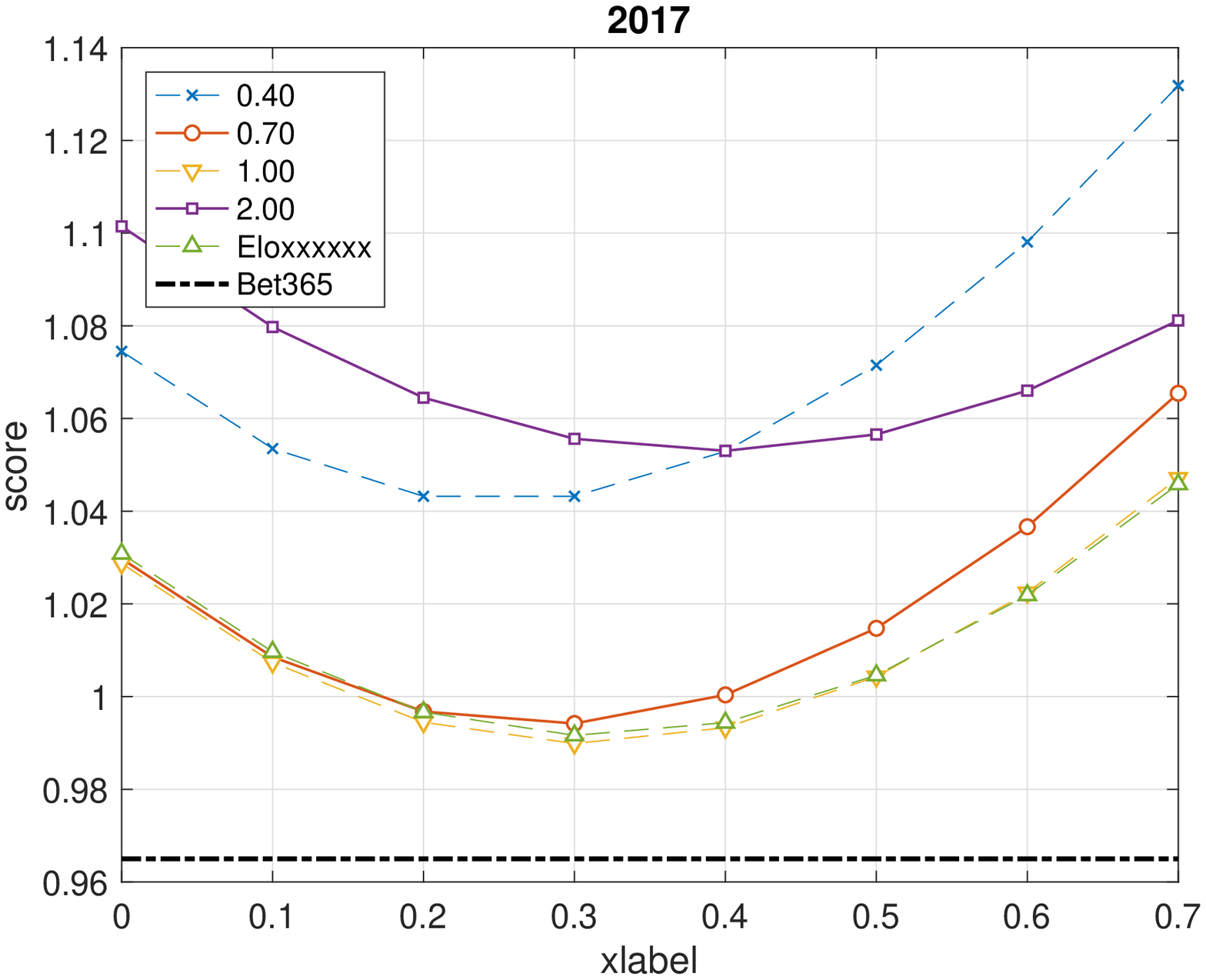} \\
\small a) &
\small b)\\
\end{tabular}
\caption{Logarithmic score, \eqref{log.score}, in second half of two seasons of English Premier League with $\sigma=600$, $\tilde{K}=0.125$, different values of $\kappa$ indicated in the legend, and varying the home-advantage parameter, $\eta$;  a)~season 2013-2014, where  $\ov{p}_\mfD=0.17$ and thus using \eqref{pt.2.kappa}, we obtain $\ov{\kappa}\approx 0.40$; b)~season 2017-2018, where  $\ov{p}_\mfD=0.26$, and thus $\ov{\kappa}\approx 0.7$. The results ``Elo+$\check{\kappa}$'' are obtained from the conventional Elo algorithm but $\check{\kappa}=1$ is used in the prediction. The results ``Bet365'' are based on the probabilities inferred from the betting odds offered by the site Bet365.}
\label{Fig:LogScore}
\end{figure}

Finally, we show the comparison across various seasons in \tabref{Tab:Score.Seasons} where, beside the score $\ov{\tnr{LS}}$ we also show the pseudo-credibility interval $(\ov{\tnr{LS}}_\tnr{low},\ov{\tnr{LS}}_\tnr{high})$; this is the the minimum-length interval in which 95\% of the data was found.\footnote{We find it more informative than derivation of credibility intervals using unknown statistics.} We observe that using $\kappa=1$ does not incur a large penalty when compared to $\kappa=0.7$ even if the latter matches closely the observed frequency of draws. The differences may be observed only for seasons where the low frequency of draws implies very small $\ov{\kappa}$, \eg in 2013-2014 and 2018-2019. 

On the other hand, the length of the credibility intervals is slightly smaller for $\kappa=1$, indicating a better prediction ``stability'' across time. Similar average results may be obtained using the Elo algorithm with $\check{\kappa}=1$ which produces also slightly larger credibility intervals.

The results obtained with this rather limited set of data stay in line with our previous theoretical discussion, indicating at the same time that no dramatic change in performance should be expected by using $\kappa$-Elo. Nevertheless, an improvement can be obtained by using the conservative value of $\kappa=1$. This recommendation is motivated by the discussion in \secref{Sec:setting.kappa} and comes at no implementation cost.

\begin{table}
\centering
\scalebox{.8}{
\begin{tabular}{c|c||c|c|c|c}
Season & $\ov{\kappa}$ & Bet365 & $\kappa$-Elo, $\kappa=0.7$  & $\kappa$-Elo, $\kappa=1$ & Elo+$\check{\kappa}$ \\
\hline
2009-2010 &   0.71  &  0.91 $\in$ (0.16,1.68)  &  0.93 $\in$ (0.19,1.64)  & 0.93 $\in$ (0.26,1.68) & 0.93 $\in$ (0.20,1.62) \\ 
2010-2011 &   0.73  &  0.97 $\in$ (0.23,1.64)  &  1.01 $\in$ (0.29,1.69)  & 1.01 $\in$ (0.35,1.66) & 1.01 $\in$ (0.32,1.70) \\ 
2011-2012 &   0.59  &  0.99 $\in$ (0.16,1.87)  &  0.98 $\in$ (0.16,1.70)  & 1.00 $\in$ (0.22,1.73) & 1.00 $\in$ (0.17,1.72) \\ 
2012-2013 &   0.73  &  0.95 $\in$ (0.24,1.66)  &  1.01 $\in$ (0.22,1.82)  & 1.01 $\in$ (0.26,1.90) & 1.00 $\in$ (0.22,1.92) \\ 
2013-2014 &   0.42  &  0.91 $\in$ (0.14,1.98)  &  0.93 $\in$ (0.17,1.86)  & 0.96 $\in$ (0.21,1.82) & 0.95 $\in$ (0.20,1.92) \\ 
2014-2015 &   0.55  &  0.96 $\in$ (0.21,1.66)  &  1.00 $\in$ (0.22,1.88)  & 1.02 $\in$ (0.27,1.94) & 1.03 $\in$ (0.23,2.00) \\ 
2015-2016 &   0.77  &  1.00 $\in$ (0.27,1.73)  &  1.02 $\in$ (0.22,1.77)  & 1.01 $\in$ (0.27,1.78) & 1.01 $\in$ (0.24,1.86) \\ 
2016-2017 &   0.57  &  0.91 $\in$ (0.15,1.91)  &  0.93 $\in$ (0.19,1.92)  & 0.94 $\in$ (0.19,1.78) & 0.94 $\in$ (0.15,1.82) \\ 
2017-2018 &   0.75  &  0.97 $\in$ (0.14,1.91)  &  0.99 $\in$ (0.18,1.78)  & 0.99 $\in$ (0.23,1.78) & 0.99 $\in$ (0.19,1.86) \\ 
2018-2019 &   0.42  &  0.91 $\in$ (0.17,1.91)  &  0.93 $\in$ (0.17,1.80)  & 0.96 $\in$ (0.21,1.87) & 0.96 $\in$ (0.17,1.96) \\ 

\end{tabular}
}
\caption{Logarithmic score $\ov{\tnr{LS}}$, \eqref{log.score}, in ten seasons of English Premier League; $\sigma=600$,  $\tilde{K}=0.125$, $\eta=0.3$. The results ``Elo+$\check{\kappa}$'' are obtained from the conventional Elo algorithm but $\check{\kappa}=1$ is used in the prediction. The results ``Bet365'' are based on the probabilities inferred from the betting odds offered by the site Bet365.}
\label{Tab:Score.Seasons}
\end{table}

\section{Conclusions}\label{Sec:Conclusions}
In this paper we were mainly concerned with explaining the rationale and mathematical foundation behind the Elo algorithm. The whole discussion may be summarized as follows:
\begin{itemize}
\item We explained that, in the binary-outcome games (win-loss), the Elo algorithm is an instance of the well-known stochastic gradient algorithm applied to solve the  \gls{ml} estimation of the rating levels. This observation already appeared in the literature, \eg \citep{Kiraly17} so it was made for completeness but also to lay ground for further discussion.
\item We have shown the implicit model behind the algorithm in the case of the games with draws. Although the algorithm has been used for decades in this type of  games, the model of the draws has not been shown, impeding, de facto, the formal prediction of their probability. We thus filled this logical gap.
\item We proposed a natural generalization of the Elo algorithm obtained from the well-known model proposed by \citep{Davidson70}; the resulting algorithm, which we call $\kappa$-Elo, has the same simplicity as the original Elo algorithm, yet provides additional parameter to adjust to the frequency of draws. By extension, we revealed that the implicit model behind the Elo algorithm assumes that the frequency of draws is equal to 50\%.
\item We briefly discussed the constraints on the relationship between the values of draw and loss probabilities for the players with similar ratings; more precisely, we postulate that, in such a case, the draw probability  should be larger than the probability of win/loss. While the  discussion on such constraints has been absent from the literature, we feel it deserves further analysis to construct suitable models and algorithms for rating. Applying these constraints to the $\kappa$-Elo algorithm yields $\kappa\ge 1$. This is clearly a limitation which will produce a mismatch between the results and the model if the frequency of draws is less than 33\%.
\item To illustrate the main concepts we have shown numerical examples based on the results of the international football games in English Premier League.
\item Finally, we conclude that, while in the past,  the Elo algorithms has satisfied to a large extent the demand for simple rating algorithms, it is still possible to provide better, more flexible, and yet simple solutions. In particular the $\kappa$-Elo is better in a sense of taking the frequency of draws into account without compromising the complexity of implementation.
\end{itemize}


\end{document}